\documentclass[preprint,12pt,authoryear]{elsarticle}

\usepackage{amssymb}
\usepackage{amsthm, amsmath}

\usepackage{mathrsfs}

\journal{Journal of Differential Equations}

\def\BM{{\rm BM}}
\def\C{{\mathbb C}}
\def\ri{{\rm i}}

\newtheorem{theorem}{Theorem}[section]
\newtheorem{lemma}[theorem]{Lemma}
\newtheorem{proposition}[theorem]{Proposition}
\newtheorem{corollary}[theorem]{Corollary}

\def\re{{\rm e}}
\def\rd{{\rm d}}
\def\e{{\rm e}}
\def\Re{{\mathbb R}}
\def\R{{\mathbb R}}

\def\A{{\cal A}}
\def\be#1{\begin{equation}\label{#1}}
\def\ee{\end{equation}}
\def\bea{\begin{eqnarray*}}
\def\eea{\end{eqnarray*}}
\def\dist{{\rm dist}}

\def\<{{\langle}}
\def\>{{\rangle}}

\def\N{\mathbb N}

\def\dist{{\rm dist}}

\def\({\left(}
\def\){\right)}
\def\cal{\mathcal}

\def\{{\left\lbrace}
\def\}{\right\rbrace}

\def\K{{\mathbb K}}

\begin{document}

\begin{frontmatter}



\title{Finite-dimensional global attractors
in Banach spaces}


\author[1]{Alexandre N. Carvalho}
\author[2]{Jos\'e A. Langa}
\author[3]{James C. Robinson}

\address[1]{Instituto de Ci\^{e}ncias Matem\'{a}ticas e de
Computa\c{c}ao, Universidade de S\~{a}o Paulo-Campus de S\~{a}o
Carlos, Caixa Postal 668, 13560-970 S\~{a}o Carlos SP, Brazil}
\address[2]{Departamento de Ecuaciones Diferenciales y An\'alisis
Num\'erico, Universidad de Sevilla, 41080-Sevilla
Spain}
\address[3]{Mathematics Institute, University of Warwick, Coventry,
CV4 7AL. U.K.}

\begin{abstract}
We provide bounds on the upper box-counting dimension of negatively
invariant subsets of Banach spaces, a problem that is easily reduced
to covering the image of the unit ball under a linear map by a
collection of balls of smaller radius. As an application of the
abstract theory we show that the global attractors of a very broad
class of parabolic partial differential equations (semilinear
equations in Banach spaces) are finite-dimensional.\end{abstract}


\end{frontmatter}

\section{Introduction}

There are now many techniques available for bounding the
(box-counting) dimension of the global attractors that have been
shown to exist for a variety of interesting models of mathematical
physics (for many examples see Temam, 1988).

To our knowledge, the earliest such result is due to Mallet-Paret
(1976), who showed that if $K$ is a compact subset of a Hilbert
space $H$, $f:H\rightarrow H$ is a continuously differentiable map,
$f(K)\supset K$ (`$K$ is negatively invariant'), and the derivative
of $f$ is everywhere equal to the sum of a compact map and a
contraction, then the upper box-counting dimension of $K$ is finite
(${\rm dim}_B(K)<\infty$). This method was subsequently generalised
by Ma\~n\'e (1981) to treat subsets of Banach spaces.

Both these methods rely on approximating the image of any ball $B$
under $f$ as the sum of a ball in a finite-dimensional subspace of
$X$ plus a `small' error. Douady \& Oesterl\'e (1980) developed a
method for subsets of $\Re^n$ that instead approximates $f(B)$ by an
ellipse, plus an error. This approach, extended to subsets of
Hilbert spaces by Constantin \& Foias (1985), produces significantly
improved bounds on the dimension and remains the most powerful to
date.

We concentrate here on the case of subsets of Banach spaces, since
the Hilbert space theory has been well developed. The problem
reduces, as we shall see, to bounding the number of balls required
to cover the image of the unit ball under a certain family of linear
maps.



We choose to work with the upper box-counting dimension (${\rm
dim}_B(\cdot)$) for two main reasons. First, it provides an upper
bound on the topological (${\rm dim}_T(\cdot)$) and Hausdorff
dimension (${\rm dim}_H(\cdot)$) (see Hurewicz \& Wallman (1941) for
more on the topological dimension and a proof that ${\rm
dim}_T(\cdot)\le {\rm dim}_H(\cdot)$; and Falconer (2003) for more
on the Hausdorff dimension and a proof that ${\rm dim}_H(X)\le {\rm
dim}_B(X)$ in general). Secondly, a set with finite upper
box-counting dimension can be embedded into a finite-dimensional
Euclidean space using a linear map (a result which fails for sets
that only have finite Hausdorff dimension) -- this was originally
proved by Ma\~n\'e in the same paper as the result which forms the
main topic of our work here, but his result has been significantly
improved in the Hilbert space case by Foias \& Olson (1996) [the
inverse is H\"older continuous] and Hunt \& Kaloshin (1999)
[explicit bounds on the H\"older exponent], and in the Banach space
case by Robinson (2009) [a Banach space version of Hunt \&
Kaloshin's result].

We now give a formal definition of the box-counting dimension. Let
$X$ be a Banach space and $K$ a compact subset of $X$. Define
$N_X(K,\epsilon)$ as
the minimum number of balls in $X$ of radius $\epsilon$ 
needed to cover $K$. The (upper) box-counting dimension ${\rm
dim}_B(K)$ of $K$ is defined by:
\begin{equation}\label{D:fracdim}
{\rm dim}_B(K)=\limsup_{\epsilon\to 0}
\frac{\log{N_X(K,\epsilon)}}{-\log\epsilon}
\end{equation}
(the lim sup in the definition is necessary; there are sets for
which the lim sup is not equal to the lim inf, see for example
Mattila, 1995). Essentially this definition extracts the exponent
$d$ from the scaling law $N_X(K,\epsilon)\sim\epsilon^{-d}$. More
rigorously,  ${\rm dim}_B(K)$ is the smallest real number such that
for any $d>{\rm dim}_B(K)$ there exists an $\epsilon_0>0$ such that
$N_X(K,\epsilon) \leqslant \epsilon^{-d}$ for all
$0<\epsilon<\epsilon_0$.


%
%

\medskip

The arguments that provide bounds on the dimension of attractors all
follow similar lines, which we can formalise in the following lemma.

\begin{lemma}\label{thebound}
Let $K$ be a compact subset of a Banach space $X$, let
$f:X\rightarrow X$ be continuously differentiable in a neigbourhood
of $K$, and let $K$ be negatively invariant for $f$, i.e.\
$f(K)\supseteq K$. Suppose that there exist $\alpha$, $0<\alpha<1$
and $M\ge 1$ such that for any $x\in K$,
\be{deriv}
N_X(Df(x)[B_X(0,1)],\alpha)\le M.
\ee
Then
\be{eq:thebound}
{\rm dim}_B(K)\le \frac{\log M}{-\log\alpha}.
\ee
\end{lemma}

\begin{proof}
First, we ensure that (\ref{deriv}) is sufficient to provide bounds
on the number of balls required to cover $f(B_X(x,r))$ when $r$ is
small enough. Since $f$ is continuously differentiable and $K$ is
compact, for any $\eta>0$ there exists an $r_0=r_0(\eta)$ such that
for any $0<r<r_0$ and any $x\in K$,
$$
f(B_X(x,r))\subseteq f(x)+Df(x)[B_X(0,r)]+B_X(0,\eta r),
$$
where $A+B$ is used to denote the set $\{a+b:\ a\in A,\ b\in B\}$.
It follows that
\be{fnow}
N_X(f[B_X(x,r)],(\alpha+\eta)r)\le M
\ee
for all $r\le r_0(\eta)$.

Now fix $\eta$ with $0<\eta<1-\alpha$, and let $r_0=r_0(\eta)$.
Cover $K$ with $N_X(K,r_0)$ balls of radius $r_0$. Apply $f$ to
every element of this cover. Since $f(K)\supseteq K$, this provides
a new cover of $K$ formed by sets all of the form $f(B_X(x,r_0))$,
for some $x\in K$. It follows from (\ref{fnow}) that each of these
images can be covered by $M$ balls of radius $(\alpha+\eta)r_0$,
ensuring that $N_X(K,(\alpha+\eta)r_0)\le MN_X(K,r_0)$. Applying
this argument $k$ times implies that
  $$
  N_X(K,(\alpha+\eta)^k r_0)\le M^k(K,r_0).
  $$
  It follows from the definition of ${\rm dim}_B(K)$ that
  $$
  {\rm dim}_B(K)\le\frac{\log M}{-\log(\alpha+\eta)},
  $$
  and since $\eta>0$ we arbitrary we obtain (\ref{eq:thebound}).
\end{proof}

The key to applying this approach is to be able to prove
(\ref{deriv}), i.e.\ to find a way of estimating the number of balls
of radius $\alpha$ required to cover $Df(x)B(0,1)$. Our argument is
based (as was that of Mallet-Paret and Ma\~n\'e) on the fact that
one can find a sharp estimate for coverings by balls in
$(\K^n,\ell^\infty)$, and then use an isomophism $T$ between
$(\K^n,\ell^\infty)$ and an $n$-dimensional linear subspace $U$ of
$X$ to relate coverings with respect to these two different norms.
We observe here that only the product $\|T\|\|T^{-1}\|$ occurs in
these estimates; this shows that covering results are related to the
Banach--Mazur distance between $(U,\|\cdot\|)$ and
$(\K^n,\ell^\infty)$, which is bounded by $\log n$ independently of
$U$. Although this follows from the powerful general result that
$d_\BM(X,Y)\le\log n$ for any two $n$-dimensional normed spaces, we
give a simple proof of the particular result that we require here.

We use our covering results to prove that a negatively invariant
compact set for a nonlinear map with a derivative that is a sum of a
strong contraction with a compact map has finite box-counting
dimension. Our proof is much simpler than that of Ma\~n\'e and our
bound on the dimension in a Banach space improves on his.

In Section \ref{Scoro} we derive interesting corollaries of the main
theorem, and consider some applications. In particular, we show that
the global attractors of a wide class of evolution equations are
finite dimensional.

\section{Coverings in Banach spaces}

In this section we provide a bound of the form
$$
N_X(T[B_X(0,1)],\alpha)\le M
$$
for linear maps $T$ that are the sum of a compact map and a
contraction. When $f:X\rightarrow X$ is such that $Df(x)$ is of this
form for all $x\in K$, and the contraction constant is uniformly
bounded over $K$, we show that (\ref{deriv}) holds uniformly over
$K$.

\subsection{The Banach--Mazur distance and coverings of balls in finite-dimensional subspaces}

The key result concerns coverings of a ball in a finite-dimensional
subspace of $X$ by balls of smaller radius. We use ideas related to
the Banach--Mazur distance to provide the estimate we need.

\subsubsection{The Banach--Mazur distance}

Let $X$ and $Y$ be normed spaces. If there exists $T\in {\mathscr
L}(X,Y)$
 that is bijective and has $T^{-1}\in {\mathscr L}(Y,X)$, we say that $X$ and $Y$ are isomorphic and that $T$ is an isomorphism between $X$ and $Y$.
The Banach-Mazur distance between two isomorphic normed spaces $X$
and $Y$, $d_{\BM}(X,Y)$,   is defined as
$$
\log\left(\inf\{ \|T\|_{\mathscr{L}(X,Y)}
\|T^{-1}\|_{\mathscr{L}(Y,X)}:\ T\in \mathscr{L}(X,Y),\
T^{-1}\in\mathscr{L}(Y,X)\}\right).
$$
 Clearly $d_\BM(X,Y)=0$ if and only if $X$ and $Y$ are isometrically isomorphic; in particular this is the case for any two separable Hilbert spaces of the same cardinality.

 It is a consequence of John's Theorem on bounding ellipses of minimal volume
  (John, 1948) that for any two $n$-dimensional real Banach spaces $U$ and $V$,
  $d_\BM(U,V)\le\log n$, see Bollob\'{a}s (1990, Theorem 4.15).
  Here we give a simple proof that if $U$ is any $n$-dimensional normed space over $\K$ ($=\R$ or $\C$), $d_\BM(U,\K^n_\infty)\le\log n$,
  where
$\K^n_\infty$ denotes $\K^n$ equipped with the $\ell^\infty$ norm:
for $\underline{z}\in\K^n$ with $\underline{z}=(z_1,\ldots,z_n)$,
$z_j\in\K$, we define
$$
\|\underline{z}\|_\infty=\max_{j=1,\ldots,n}\|z_j\|_\K.
$$

In order to show that $d_\BM(U,\K^n_\infty)\le\log n$ we will use an
Auerbach basis for $X$. The proof of the existence of such a basis
when $X$ is real is standard (see Bollob\'{a}s, 1990, Theorem 4.13,
for example); we give a proof of the complex case in an appendix.

\begin{lemma}\label{auerbach}
 Let $X$ be a $n$-dimensional normed vector space (which may be real or complex). Then, there exists a basis $B=\{x_1,\cdots, x_n\}$ for $X$ and a basis $B^*=\{f_1,\cdots, f_n\}$ for $X^*$ with $\|x_i\|_{X}=\|f_i\|_{X^*}=1$ ($i=1,\ldots,n$) such that $f_i(x_j)=\delta_{i\! j}$, $i,j=1,\ldots,n$.
\end{lemma}

Given this lemma the proof of the following proposition is
straightforward.

\begin{proposition}
Let $U$ be an $n$-dimensional Banach space over $\K=\R$ or $\C$.
Then $d_\BM(U,\K^n_\infty)\le\log n$.
\end{proposition}

\noindent Note that Ma\~{n}\'{e}'s paper contains a proof of the
bound $d_{\BM}(U,\Re^n_\infty)\leq \log(n2^n)$ when $U$ is real.

\begin{proof}
Let $\{x_1,\ldots,x_n\}$ be an Auerbach basis for $U$, and
$\{f_1,\ldots,f_n\}$ the corresponding basis for $U^*$. Define a map
$J:\K^n_\infty\rightarrow U$ by setting
$$
J(\underline{z})=\sum_{j=1}^n z_jx_j.
$$
Then
$$
\|J(\underline{z})\|_X=\left\|\sum_{j=1}^n
z_jx_j\right\|_X\le\sum_{j=1}^n|z_j|\le n\|\underline{z}\|_\infty,
$$
and so
$$
\|J\|_{\mathscr{L}(\K^n_\infty,U)}\le n.
$$
On the other hand, if $x=\sum_{j=1}^n z_jx_j\in U$ with $\|x\|_X\le
1$ then since $z_j=f_j(x)$,
$$
\|J^{-1}(x)\|_\infty=\|\underline{z}\|_\infty=\max_{j=1,\ldots,n}|z_j|=\max_{j=1,\ldots,n}
|f_j(x)| \leqslant \|x\|_X,
$$
which implies that
$$
\|J^{-1}\|_{\mathscr{L}(X,X_B^\infty)} \leqslant 1.
$$
\end{proof}

\subsubsection{The covering lemma}

We use this result to prove our covering lemma:

\begin{lemma}\label{cover:Banach}
If $U$ is an $n$-dimensional subspace of a real Banach space $X$,
then
\begin{equation}\label{eq:CX}
N_X(B_U(0,r),\rho)\le (n+1)^n\left(\frac{r}{\rho}\right)^n\qquad
0<\rho\le r,
\end{equation}
where the balls in the cover can be taken to have centres in $U$.
The same result holds in a complex Banach space if one replaces the
right-hand side of (\ref{eq:CX}) with its square.
\end{lemma}

\begin{proof} Assume first that $\K=\R$. Since $U$ and $\Re^n_\infty$ are $n$-dimensional, $d_\BM(U,\Re^n_\infty)\le\log n$: in particular, there exists a linear isomorphism $T:\Re^n_\infty\rightarrow U$ such that $\|T\|\|T\|^{-1}\le n$. Since
$$
B_U(0,r)=TT^{-1}(B_U(0,r))\subseteq
T(B_{\Re^n_\infty}(0,\|T^{-1}\|r)),
$$
and $B_{\Re^n_\infty}(0,\|T^{-1}\|\, r)$ can be covered by
$$
\left(1+\frac{\|T^{-1}\|r}{\rho/\|T\|}\right)^n=\left(1+\|T\|\|T^{-1}\|\,\frac{r}{\rho}\right)^n\le
\left(1+n\,\frac{r}{\rho}\right)^n\le
(n+1)^n\,\left(\frac{r}{\rho}\right)^n
$$
balls in $\Re^n_\infty$ of radius $\rho/\|T\|$, it follows that
$B_U(0,r)$ can be covered by the same number of $U$-balls of radius
$\rho$. If $X$ is complex one requires $(1+(a/b))^{2n}$ $b$-balls in
$\C^n_\infty$ to cover a ball of radius $a$.
\end{proof}

\subsection{Coverings of $T[B_X(0,1)]$ via finite-dimensional approximations}

We now have good estimates for the coverings of balls in
finite-dimensional linear subspaces, but we want to cover the images
of balls under linear maps. In order to do this we show that, given
a linear map $T$ that is the sum of a compact map and a contraction,
$T[B_X(0,1)]$ can be well-approximated by $T[B_Z(0,1)]$, where $Z$
is a finite-dimensional subspace of $X$.

We denote by ${\mathscr L}(X)$ the space of bounded linear
transformations from $X$ into itself, by $\mathcal{K}(X)$ the closed
subspace of ${\mathscr L}(X)$ consisting of all compact linear
transformations from $X$ into itself, and define
\begin{equation}\label{Dell-lambda}
{\mathscr L}_\lambda(X)=\{T\in {\mathscr L}(X): T=L+C, \hbox{ with }
C\in \mathcal{K}(X) \hbox{ and } \|L\|_{{\mathscr L}(X)}<\lambda\}.
\end{equation}

By $\dist(A,B)$ we denote the Hausdorff semi-distance between $A$
and $B$,
$$
\dist(A,B)=\sup_{a\in A}\left(\inf_{b\in B}\|a-b\|_X\right).
$$

\begin{lemma}\label{Lll-aux}
Let $X$ be a Banach space and $T\in{\mathscr L}_{\lambda/2}(X)$.
Then there exists a finite-dimensional subspace $Z$ of $X$ such that
\be{itsZ}
{\rm dist}(T[B_X(0,1)],T[B_Z(0,1)])<\lambda.
\ee
We denote by $\nu_\lambda(T)$ the minimum $n\in\N$ such that
(\ref{itsZ}) holds for some $n$-dimensional subspace of $X$.
\end{lemma}

\begin{proof}
Write $T=L+C$, where $C\in\mathcal{K}(X)$ and $L\in{\mathscr L}(X)$
with $\|L\|_{{\mathscr L}(X)}<\lambda/2$. We show first that for any
$\epsilon>0$ there is a finite-dimensional subspace $Z$ such that
\begin{equation}\label{eq:e-approx}
{\rm dist}(C[B_X(0,1)],C[B_Z(0,1)])<\epsilon.
\end{equation}
Suppose that this is not the case. Choose some $x_1\in X$ with
$\|x_1\|_X=1$, and let $Z_1={\rm span}\{x_1\}$. Then
$$
{\rm dist}(C[B_X(0,1)],C[B_{Z_1}(0,1)])\ge\epsilon,
$$
and so there exists an $x_2\in X$ with $\|x_2\|_X=1$ such that
$$
\|Cx_2-Cx_1\|_X\geqslant \epsilon.
$$
With $Z_2={\rm span}\{x_1,x_2\}$, one can find an $x_3$ with
$\|x_3\|_X=1$ such that
$$
\|Cx_3-Cx_1\|_X\geqslant \epsilon\ \hbox{ and }\ \|Cx_3-Cx_2\|_X
\geqslant \epsilon.
$$
Continuing inductively one can construct in this way a sequence
$\{x_j\}$ with $\|x_j\|=1$ such that
$$
\|Cx_i-Cx_j\|_X\ge\epsilon\qquad i\neq j,
$$
contradicting the compactness of $C$.

\medskip

Now let $\tilde\lambda<\lambda$ be such that $2\|L\|_{{\mathscr
L}(X)}<\tilde\lambda<\lambda$, and choose $Z$ using the above
argument so that
$$
{\rm dist}(C[B_X(0,1)],C[B_Z(0,1)])<\lambda-\tilde\lambda.
$$
If $x\in B_X(0,1)$ and $z\in B_Z(0,1)$, then
$$
\|Tx-Tz\|_X\leq \|L(x-z)\|_{X}+\|Cx-Cz\|_{X}\leq \tilde\lambda +
\|Cx-Cz\|_{X}
$$
Hence,
\begin{align*}
{\rm dist}(T[B_X(0,1)],T[B_Z(0,1)]) &\leqslant \tilde\lambda + {\rm
dist}(C[B_X(0,1)],C[B_Z(0,1)])\\
&< \lambda.
\end{align*}

This completes the proof. \end{proof}
%
%
%
%
%

\subsection{Uniform estimates for $x\in K$}

\begin{theorem}[after Ma\~n\'e, 1981]\label{TMane-main-improved}
Let $X$ be a Banach space, $U\subset X$ an open set, and $f:U\to X$
a continuously differentiable map. Suppose that $K$ is a compact set
and assume that for some $\lambda$ with $0<\lambda<\frac{1}{2}$,
$$
Df(x) \in {\mathscr L}_{\lambda/2}(X)\qquad\mbox{for all}\quad x\in
K.
$$
Then $n=\sup_{x\in K} \nu_\lambda (Df(x))$ and $D=\sup_{x\in
K}\|Df(x)\|$ are finite, and
\begin{equation}\label{Edh-kni-improved}
N(Df(x)[B_X(0,1)],2\lambda) \leqslant
\left[(n+1)\,\frac{D}{\lambda}\right]^{\alpha n}\qquad\mbox{for
all}\quad x\in K,
\end{equation}
where $\alpha=1$ if $X$ is real and $\alpha=2$ if $X$ is complex. It
follows that
\begin{equation}\label{eq:upperbd}
{\rm dim}_B(K) \leqslant \alpha
n\,\left\lbrace\frac{\log((n+1)D/\lambda)}
{-\log(2\lambda)}\right\rbrace,
\end{equation}
\end{theorem}



\begin{proof}
First we show that $n=\sup_{x\in K}\nu_\lambda(Df(x))$ is finite.
For each $x\in K$, there exists a finite-dimensional linear subspace
$Z_x$ such that
$$
\dist(Df(x)[B_X(0,1)],Df(x)[B_{Z_x}(0,1)])<\lambda.
$$
Since $Df(\cdot)$ is continuous, it follows that there exists a
$\delta_x>0$ such that
$$
\dist(Df(y)[B_X(0,1)],Df(y)[B_{Z_x}(0,1)])<\lambda
$$
for all $y\in B_X(x,\delta_x)$,
i.e.~$\nu_\lambda(y)\le\nu_\lambda(x)$ for all such $y$. The open
cover of $K$ formed by the union of $B_X(x,\delta_x)$ over $x$ has a
finite subcover, from whence it follows that $n<\infty$.

Now, since $n=\sup_{x\in K} \nu_\lambda (Df(x)) < \infty$, for each
$x\in K$ there is a subspace $Z_x$ of $X$ with ${\rm dim}(Z_x)
\leqslant n$ such that
$$
{\rm dist}(Df(x)[B_X(0,1)],Df(x)[B_{Z_x}(0,1)])<\lambda.
$$
For ease of notation we now drop the $x$ subscript on $Z_x$, and
write $T=Df(x)$.

Noting that $T(Z)$ is also an $n$-dimensional subspace of $X$, one
can use Lemma \ref{cover:Banach} to cover the ball
$B_{T(Z)}(0,\|T\|)$ with balls $B_X(y_i,\lambda)$, $1\leqslant
i\leqslant k$, such that $y_i\in B_X(0,\|T\|)$ for each $i$ and
$$
k \leqslant \left[(n+1)\,\frac{\|T\|}{\lambda}\right]^{\alpha n}.
$$
Thus
\begin{equation}\label{cover_TZ}
T[B_{Z}(0,1)]\subseteq B_{T(Z)}(0,\|T\|)=B_X(0,\|T\|)\cap
T(Z)\subseteq \bigcup_{i=1}^k B_X(y_i,\lambda).
\end{equation}
We complete the proof by showing that
$$
\bigcup_{i=1}^k  B_X(y_i,{2\lambda})\supseteq T[B_X(0,1)].
$$
Indeed, if $x\in B_X(0,1)$ then it follows from \eqref{itsZ} that
there is a $y\in T[B_{Z}(0,1)]$ such that $\|Tx-y\|_X<\lambda$.
Since $y\in T[B_{Z}(0,1)]$, it follows from (\ref{cover_TZ}) that
$\|y-y_i\|_X\leqslant \lambda$ for some $i\in\{1,\ldots,k\}$, and so
$$
\|Tx-y_i\|_X \leqslant \|Tx-y\|_X + \|y-x_i\|_X< 2\lambda,
$$
i.e.~$x\in B_X(y_i,2\lambda)$.

The result now follows as stated since $n$ is uniform over $x\in K$.
 \end{proof}

Some immediate improvement of the above is possible if we work in a
Hilbert space. A result due to Cheypzhov \& Vishik (2002) (Lemma
III.2.1) shows that in $\R^n$ we need no more than $4^n(r/\rho)^n$
balls of radius $\rho$ to cover a ball of radius $r>\rho$. This
implies that one can replace the factor $(n+1)^n$ in (\ref{eq:CX})
by $7^n$, and obtain a corresponding improvement in the bound of
Theorem \ref{TMane-main-improved}. We do not pursue this direction
further here, since the use of ellipses in place of balls (an
approach initiated by Douady \& Oesterl\'e, 1980, and developed
further by Constantin \& Foias, 1985) leads to significant
improvements on the possible dimension estimates and is now the
standard approach.

\section{Corollaries \& Applications}\label{Scoro}

\subsection{When $Df\in{\mathscr L}_1(X)$}

The following corollary can be found in Hale, Maghala\~es, \& Oliva
(2002):

\begin{corollary}\label{Cmain}
Suppose that $X$ is a Banach space, $U\subset X$ an open set, and
$f:U\rightarrow X$ a continuously differentiable map. Suppose that
$K\subset U$ is a compact set such that $f(K)\supseteq K$, and that
$Df(x)\in{\mathscr L}_1(X)$ for all $x\in K$. Then ${\rm
dim}_B(K)<\infty$.
\end{corollary}

\begin{proof}
It follows from an argument similar to that used in the proof of
Theorem \ref{TMane-main-improved} to show that $n<\infty$ that in
fact there exists an $\alpha<1$ such that $Df(x)\in{\mathscr
L}_\alpha(X)$ for all $x\in K$. Note that
  $$
  D[f^p]=Df(f^{p-1}(x))\circ\cdots\circ Df(x),
  $$
  and that if $C_{i}\in{\mathcal K}(X)$ and $L_{i}\in{\mathscr
  L}(X)$, $i=1,2$, then
  $$
  (C_1+L_1)\circ(C_2+L_2)=\underbrace{[C_1\circ C_2+C_1\circ L_2+L_1\circ
  C_2]}_{\in\ {\mathcal K}(X)}+L_1\circ L_2,
  $$
  it follows that if $Df(x)\in{\mathscr
  L}_\alpha(X)$ with $\alpha<1$ then $[D(f^p)](x)\in{\mathscr
  L}_{\alpha^p}(X)$. It follows that for $p$ large enough, $D(f^p)(x)\in{\mathscr
  L}_\lambda$ for some $\lambda<1/4$, for every $x\in K$. One can
  now apply Theorem \ref{TMane-main-improved} to $f^p$ in place of
  $f$ (noting that $f^p(K)\supseteq K$) to deduce
  that $d_f(K)<\infty$.
\end{proof}

\subsection{When $D_xT$ has finite rank}

\begin{corollary} Let $X$ be a Banach space and
assume that $T\in {\mathcal C}^1(X)$, $K$ is a compact set such that
$T(K)=K$, and $D_xT$ has finite rank $\nu(x)$ with $\sup_{x\in K}
\nu(x):=\nu<\infty$. Then,
$$
{\rm dim}_B(K)\leqslant \nu.
$$
\end{corollary}

\begin{proof} Clearly, for each $\lambda>0$ and $x\in K$, $D_xT\in
{\mathscr L}_\frac{\lambda}{2}(X)$ for all $\lambda>0$.
Consequently, for each $0<\lambda <\frac12$,
$$
{\rm dim}_B(K) \leqslant \nu\frac{\log\left((\nu+1)
\frac{D}{\lambda}\right)}{\log(1/2\lambda)}.
$$
Taking the limit as $\lambda\to 0$ we have that ${\rm dim}_B(K)
\leqslant \nu$.\end{proof}

\subsection{An ordinary differential equation}

\begin{corollary}
Let $f:\R^n\to \R^n$ be a continuously differentiable function.
Assume that the semigroup $\{S(t):t\geqslant 0\}$ in $\R^n$
associated to the ordinary differential equation
\begin{equation*}
\dot x = f(x) \qquad x(0)=x_0\in \R^n.
\end{equation*}
has a global attractor $\mathcal{A}$. If ${\rm rank}(D_x f)\leqslant
k\leqslant n$ for all $x\in \A$, then ${\rm dim}_B(\A)\leqslant k$.
\end{corollary}

In particular, if $f:\R^k\to\R^k$, $\beta>0$ and there is a constant
$M>0$ such that $f(x)\cdot x <0$ for $\|x\|_{\R^k}\geqslant M$, then
the semigroup $\{S(t):t\geqslant 0\}$ associated to
\begin{equation*}
\begin{split}
\frac{\rd}{\rd t}&\left(\begin{matrix} x \\
y\end{matrix}\right)=\left(\begin{matrix}
\, 0\, & \, I\, \\
0 & -\beta\end{matrix}\right)
\left(\begin{matrix} x \\
y\end{matrix}\right)+ \left(\begin{matrix} 0 \\
f(x)\end{matrix}\right)\\
&\left(\begin{matrix} x \\
y\end{matrix}\right)(0)=\left(\begin{matrix} x_0 \\
y_0\end{matrix}\right)
\end{split}
\end{equation*}
has a global attractor $\A$ in $\R^k\times\R^k$ with ${\rm
dim}_B(\A)\leqslant k$.

\subsection{Semilinear equations in Banach spaces}

\begin{corollary}\label{Coro-sl}
Let $A:D(A)\subset X\to X$ be a sectorial operator with ${\rm
Re}\sigma(A)>0$. If $f:X^\alpha\to X$ is continuously differentiable
and Lipschitz continuous in bounded subsets of $X^\alpha$ and the
semigroup $\{S(t):t\geqslant 0\}$ in $X^\alpha$ associated to the
abstract parabolic problem
\begin{equation*}
\dot x + Ax = f(x) \qquad\mbox{with}\qquad x(0)=x_0\in X^\alpha
\end{equation*}
has a global attractor $\mathcal{A}$ and either $\re^{-At}$ is
compact for each $t>0$ or $f_x\in \mathcal{K}(X^\alpha,X)$ is
compact for each $x\in \A$, then ${\rm dim}_B(\A)<\infty$.
\end{corollary}

\begin{proof} For $x\in \A$, let
$$
S(t)x=\re^{-At}x + \int_0^t \re^{-A(t-s)} f(S(s)x)ds
$$
hence, the derivative $S_x(t)\in {\mathscr L}(X^\alpha)$ with
respect to $x$ of $S(t)$ at $x$ satisfies
$$
S_x(t)=\re^{-At} + \int_0^t \re^{-A(t-s)} f'(S(s)x)S_x(s) ds.
$$
Hence, for $t$ suitably large, the hypothesis of Corollary
\ref{Cmain} are satisfied and the result follows.

\end{proof}

Now we show how a rough estimate on the dimension of the attractor
can be obtained. First note that, if $A:D(A)\subset X\to X$ is a
sectorial operator with compact resolvent, $X^\beta$,
$\beta\geqslant 0$ denotes the fractional power spaces associated to
$A$, $\alpha>0$, there is a sequence of finite rank projections
$\{P_n\}_{n\in \N}$ and sequences of positive real numbers
$\{\lambda_n\}_{n\in \N}$ and $\{M_n\}_{n\in \N}$ such that
\begin{equation}\label{Eadmissible}
\|\re^{-At}(I-P_n)\|_{{\mathscr L}(X^\gamma,X^\beta)}\leqslant M_n
t^{-(\beta-\gamma)}\re^{-\lambda_n t},\ t\geqslant 0,\
0\leqslant\gamma\leqslant \beta\leqslant \alpha.
\end{equation}
We say that $A$ is an admissible sectorial operator if it is
sectorial and there is a sequence  $\{\lambda_n\}_{n\in \N}$ and
$M>0$ such that \eqref{Eadmissible} with $M_n=M$ for all $n\in \N$.
It is not difficult to see that, if $A$ is an admissible sectorial
operator, then $A$ has compact resolvent.
$$
\|S_x(t)\|_{{\mathscr L}(X^\alpha)}\leqslant M + MN \int_0^t
(t-s)^{-\alpha} \|S_x(s)\|_{{\mathscr L}(X^\alpha)}ds,
$$
where
$$
N=\sup\{\|f'(x)\|_{{\mathscr L}(X^\alpha,X)}: x\in \A\}
$$
It follows from the generalized Gronwall inequality that
$$
\|S_x(t)\|_{{\mathscr L}(X^\alpha)}\leqslant \frac{\bar M
}{1-\alpha} \re^{(MN\Gamma(1-\alpha))^{1/(1-\alpha)} \, t}.
$$
Now, if $Q_n=(I-P_n)$,
$$
\|Q_nT_x(t)\|_{{\mathscr L}(X^\alpha)}\leqslant M\re^{-\lambda_n t}
+ MN \int_0^t (t-s)^{-\alpha}\re^{-\lambda_n(t-s)}
\|S_x(s)\|_{{\mathscr L}(X^\alpha)}ds
$$
and
\begin{equation*}
\begin{split}
&\|Q_nT_x(t)\|_{{\mathscr L}(X^\alpha)}\\
&\leqslant M\re^{-\lambda_n t} + \frac{\bar M
MN}{1-\alpha}\re^{(MN\Gamma(1-\alpha))^{1/(1-\alpha)}} \int_0^t
(t-s)^{-\alpha}\re^{-(\lambda_n
+(MN\Gamma(1-\alpha))^{1/(1-\alpha)})(t-s)}ds
\\
&\leqslant M\re^{-\lambda_n t} + \frac{\bar M
MN}{1-\alpha}\re^{(MN\Gamma(1-\alpha))^{1/(1-\alpha)}t} \int_0^t
u^{-\alpha}\re^{-(\lambda_n
+(MN\Gamma(1-\alpha))^{1/(1-\alpha)})u}du
\\
&\leqslant M\re^{-\lambda_n t} + \frac{\bar M MN}{1-\alpha}\
\frac{\re^{(MN\Gamma(1-\alpha))^{1/(1-\alpha)}t}
\Gamma(1-\alpha)}{\lambda_n
+(MN\Gamma(1-\alpha))^{1/(1-\alpha)}}=\Lambda_n(t), \ t\geqslant 0.
\end{split}
\end{equation*}

From the admissibility of $A$ and \eqref{Eadmissible},
$\|Q_n\|_{{\mathscr L}(X^\alpha)}\leqslant M$, for all $n\in \N$.
Hence,
$$
\|Q_nT_x(t)Q_n\|_{{\mathscr L}(X^\alpha)} \leqslant M\Lambda(t), \
t\geqslant 0.
$$

Choose $t=1$ and $n_0\in \N$ such that $\Lambda(1)<\lambda<
\frac{1}{4}$. If $F=S(1)$, $L=Q_{n_0}S(1)$ and $C=P_{n_0}S(1)$, then
$\A$ is invariant under $F$. Furthermore, $F_x= L_x + C_x$ with
$L_x= Q_{n_0}S_x(1)$ and $C_x = P_{n_0}S_x(1)$ and if $Z_x=R(C_x)$
and $W_x$ is a subspace of $X$ such that $C_x:W_x\to Z_x$ is an
isomorphism, $F_x\in {\mathscr L}_{\lambda/2}(X)$ for all $x\in
{\mathcal A}$ and for some $\lambda<\frac12$. In addition,
$$
\nu=\sup_{x\in \A}{\rm dim}(Z_x)\leqslant {\rm dim}(R(P)).
$$
This proves that all the assumptions of Theorem
\ref{TMane-main-improved} are satisfied and consequently
\begin{equation}\label{Edim-est-improved}
{\rm dim}_B(\A)\leqslant \nu\frac{\log \left( (\nu+1) \,
\frac{D}{\lambda}\right)}{\log(1/2\lambda)}<\infty.
\end{equation}

\section{Conclusion}

As specific examples of the result in Section \ref{Scoro} (a) we
mention the attractor for the damped wave equation with critical
exponent (see Arrieta et al (1992)) and as an application of the
results in Section \ref{Scoro} (d) we mention the attractors of
dissipative parabolic equations in $L^p(\Omega)$, $W^{1,p}(\Omega)$
as in Arrieta et al (2000)  or the Navier-Stokes Equation in space
dimension 2 as in Temam (1988).

As an example of a problem which does not define a semigroup in a
Hilbert space we mention that of Arrieta et al (submitted for
publication), for Section \ref{Scoro} (d). For Section \ref{Scoro}
(a), the examples are the  attractors for functional differential
equations for which the natural phase space is not a Hilbert space
(see Hale et al., 2002, for example).

\section*{Appendix: Auerbach bases in a complex Banach spaces}

In this appendix we give a proof of the existence of an Auerbach
basis in a complex $n$-dimensional Banach space. The result is
standard in a real Banach space, and we use the real version in our
proof.

\begin{lemma}\label{auerbach}
 Let $X$ be a $n$-dimensional complex normed vector space. Then, there exists a basis $B=\{x_1,\cdots, x_n\}$ for $X$ and a basis $B^*=\{f_1,\cdots, f_n\}$ for $X^*$ with $\|x_i\|_{X}=\|f_i\|_{X^*}=1$ ($i=1,\ldots,n$) such that $f_i(x_j)=\delta_{i\! j}$, $i,j=1,\ldots,n$.
\end{lemma}

\begin{proof} Given a basis $B_0=\{y_1,\cdots, y_n\}$ of $X$ we consider the real space 
$X_2$ given as the linear span of over $\R$ of
$$
B_2=\{y_1,\ldots,y_n,\ri y_1,\ldots,\ri y_n\},
$$
equipped with the norm $\|z\|_{X_2}=\|z\|_X$. We now apply the real
version of the result to $X_2$, to produce a basis
$\{x_1,\ldots,x_{2n}\}$ for $X_2$ and
$\{\varphi_1,\ldots,\varphi_{2n}\}$ for $X_2^*$ such that
$\varphi_j(x_k)=\delta_{jk}$ and $\|\varphi_j\|=\|x_j\|=1$.

Since $x_j\in X$ and the $\{x_j\}$ must span $X$ over $\C$, we can
reorder and relabel the $\{x_j\}$ (and the corresponding
$\{\varphi_j\}$) so that $\{x_1,\ldots,x_n\}$ span $X$ over $\C$. It
follows that
$$
\{x_1,\ldots,x_n,\ri x_1,\ldots,\ri x_n\}
$$
span $X$ over $\R$. Furthermore, each $\ri x_j$ must be a linear
combination of $\{x_{n+1},\ldots,x_{2n}\}$, since it follows from
the fact that $\{x_1,\ldots,x_n\}$ form a basis for $X$ over $\C$
that $\{\ri x_1,x_2,x_n\}$ are linearly independent (over $\C$, so
certainly over $\R$).

For $k=1,\ldots,n$ we define an element $f_k\in X^*$ via
$$
f_k(z)=\varphi_k(z)-\ri \varphi_k(\ri z)\qquad z\in X,
$$
where in order to interpret $\varphi_k(z)$ we consider $z$ as an
element of $X_2$ (expand in terms of the basis $\{x_1,\ldots,x_n\}$
and split the real and imaginary parts of the coefficients). We now
show that the $\{f_k\}$ have the properties we require.

First, note that for $1\le k,j\le n$,
$$
f_k(x_j)=\varphi_k(x_j)-\ri \varphi_k(\ri x_j)=\delta_{kj},
$$
since $\ri x_j$ is a linear combination of
$\{x_{n+1},\ldots,x_{2n}\}$, and thus $\varphi_k(\ri x_j)=0$. All
that remains is to show that $\|x_k^*\|=1$. To do this we follow the
standard argument (e.g.\ Yosida, 1980), writing
$f_k(z)=r\e^{-\ri\theta}$. Then
$$
|f_k(z)|=\e^{\ri\theta}f_k(z)=f_k(\e^{\ri\theta}z),
$$
so that $f_k(\e^{\ri\theta}z)$ is real and positive. It follows that
$$
|f_k(z)|=f_k(\e^{\ri\theta}z)=\varphi_k(\e^{\ri\theta}z)\le\|\e^{\ri\theta}z\|_{X_2}\le\|\e^{\ri\theta}z\|_X=\|x\|_X,
$$
and the lemma is proved.
\end{proof}

%
%
%
%
%
%
%
%

\bigskip

\section*{Acknowledgements}
Part of this work
has been carried out while ANC and JAL were visiting the Warwick
Mathematics Institute, University of Warwick, UK. They both wish to
acknowledge the great hospitality of the people in this institution.
The paper was completed while all three authors were visiting the de
Giorgi Centre in Pisa; we are very grateful for the opportunity this
visit gave us to continue our collaboration. ANC was partially
supported by grants CNPq 302022/2008-2 and 451761/2008-1, CAPES/DGU
267/2008 and FAPESP 2008/53094 and 2008/55516-3, Brazil; JAL by
Ministerio de Ciencia e Innovaci\'on grant MTM2008-0088, and Junta
de Andaluc\'{\i}a grants P07-FQM-02468 and FQM314; JCR wishes to
acknowledge the support of the EPSRC, grant EP/G007470/1

\end{document}